\documentclass[12pt]{amsart}

\usepackage{amsmath,amssymb,amsbsy,amsfonts,amsthm,latexsym,
                        amsopn,amstext,amsxtra,euscript,amscd,mathrsfs,color,bm,cite}
                       
\usepackage{float} 
\usepackage[english]{babel}
\usepackage{mathtools}
\usepackage{todonotes}
\usepackage{url}
\usepackage[colorlinks,linkcolor=blue,anchorcolor=blue,citecolor=blue,backref=page]{hyperref}
\usepackage{enumitem}

\usepackage[norefs,nocites]{refcheck}

\begin{document}

\newtheorem{theorem}{Theorem}
\newtheorem{lemma}[theorem]{Lemma}
\newtheorem{claim}[theorem]{Claim}
\newtheorem{cor}[theorem]{Corollary}
\newtheorem{prop}[theorem]{Proposition}
\newtheorem{definition}{Definition}
\newtheorem{question}[theorem]{Question}
\newtheorem{conj}[theorem]{Conjecture}
\newcommand{\hh}{{{\mathrm h}}}

\theoremstyle{remark}
\newtheorem{rem}{Remark}

\numberwithin{equation}{section}
\numberwithin{theorem}{section}
\numberwithin{table}{section}

\numberwithin{figure}{section}

\def\sssum{\mathop{\sum\!\sum\!\sum}}
\def\ssum{\mathop{\sum\ldots \sum}}
\def\iint{\mathop{\int\ldots \int}}

%\def\squareforqed{\hbox{\rlap{$\sqcap$}$\sqcup$}}
%\def\qed{\ifmmode\squareforqed\else{\unskip\nobreak\hfil
%\penalty50\hskip1emrll\nobreak\hfil\squareforqed
%\parfillskip=0pt\finalhyphendemerits=0\endgraf}\fi}%%

%  use the AMS-Euler Fraktur fonts
%%%%%%%%%%%%%%%%%%%%%%%%%%%%%%%%%%
\newfont{\teneufm}{eufm10}
\newfont{\seveneufm}{eufm7}
\newfont{\fiveeufm}{eufm5}
%%%%%%%%%%%%%%%%%%%%%%%%%%%%%%%%%
%
%  allow automatic size selection in math mode
%
%%%%%%%%%%%%%%%%%%%%%%%%%%%%%%%%%
\newfam\eufmfam
     \textfont\eufmfam=\teneufm
\scriptfont\eufmfam=\seveneufm
     \scriptscriptfont\eufmfam=\fiveeufm
%%%%%%%%%%%%%%%%%%%%%%%%%%%%%%%%%
%
%  \frak works on a single symbol at a time...
%
\def\frak#1{{\fam\eufmfam\relax#1}}

\newcommand{\bflambda}{{\boldsymbol{\lambda}}}
\newcommand{\bfmu}{{\boldsymbol{\mu}}}
\newcommand{\bfxi}{{\boldsymbol{\xi}}}
\newcommand{\bfrho}{{\boldsymbol{\rho}}}

\def\fA{{\mathfrak A}}
\def\fB{{\mathfrak B}}
\def\fC{{\mathfrak C}}
\def\fI{{\mathfrak I}}
\def\fJ{{\mathfrak J}}
\def\fK{{\mathfrak K}}
\def\fM{{\mathfrak M}}
\def\fS{{\mathfrak S}}
\def\fW{{\mathfrak W}}
\def\fU{{\mathfrak U}}

\def \balpha{\bm{\alpha}}
\def \bbeta{\bm{\beta}}
\def \bgamma{\bm{\gamma}}
\def \blambda{\bm{\lambda}}
\def \bchi{\bm{\chi}}
\def \bphi{\bm{\varphi}}
\def \bpsi{\bm{\psi}}
\def \bomega{\bm{\omega}}
\def \btheta{\bm{\vartheta}}

\def \bzeta{\bm{\zeta}}
\def \bxi{\bm{\xi}}

\newcommand{\vect}[1]{\mathbf{#1}}

%\def\squareforqed{\hbox{\rlap{$\sqcap$}$\sqcup$}}
%\def\qed{\ifmmode\squareforqed\else{\unskip\nobreak\hfil
%\penalty50\hskip1emrll\nobreak\hfil\squareforqed
%\parfillskip=0pt\finalhyphendemerits=0\endgraf}\fi}

%%%%%%%%%%%%%%%%%%%%%%%%%
% Alphabet calligraphie %
%%%%%%%%%%%%%%%%%%%%%%%%%
\def\cA{{\mathscr A}}
\def\cB{{\mathscr B}}
\def\cC{{\mathscr C}}
\def\cD{{\mathscr D}}
\def\cE{{\mathscr E}}
\def\cF{{\mathscr F}}
\def\cG{{\mathscr G}}
\def\cH{{\mathscr H}}
\def\cI{{\mathscr I}}
\def\cJ{{\mathscr J}}
\def\cK{{\mathscr K}}
\def\cL{{\mathscr L}}
\def\cM{{\mathscr M}}
\def\cN{{\mathscr N}}
\def\cO{{\mathscr O}}
\def\cP{{\mathscr P}}
\def\cQ{{\mathscr Q}}
\def\cR{{\mathscr R}}
\def\cS{{\mathscr S}}
\def\cT{{\mathscr T}}
\def\cU{{\mathscr U}}
\def\cV{{\mathscr V}}
\def\cW{{\mathscr W}}
\def\cX{{\mathscr X}}
\def\cY{{\mathscr Y}}
\def\cZ{{\mathscr Z}}
\newcommand{\rmod}[1]{\: \mbox{mod} \: #1}

\def\cg{{\mathcal g}}

\def\vr{\mathbf r}

\def\e{{\mathbf{\,e}}}
\def\ep{{\mathbf{\,e}}_q}
\def\eq{{\mathbf{\,e}}_q}

\def\Tr{{\mathrm{Tr}}}
\def\Nm{{\mathrm{Nm}}}

 \def\SS{{\mathbf{S}}}

\def\lcm{{\mathrm{lcm}}}

\def\fl#1{\left\lfloor#1\right\rfloor}
\def\rf#1{\left\lceil#1\right\rceil}

\def\mand{\qquad \mbox{and} \qquad}

\newcommand{\commB}[2][]{\todo[#1,color=blue!60]{B: #2}}
\newcommand{\commI}[2][]{\todo[#1,color=green!60]{I: #2}}
\newcommand{\commII}[2][]{\todo[#1,color=red!60]{I: #2}}
\newcommand{\commJ}[2][]{\todo[#1,color=red!60]{Jr: #2}}

%%%%%%%%%%%%%%%%%%%%%%%%%%%%%%%%%%%%%%%%%%%%%%%%%%%%%%%%
%%%%%%%%%%%%%%%%%%%%%%%%%%%%%%%%%%%%%%%%%%%%%%%%%%%%%%%%
%%%%%%%%%%%%%%%%%%%%%%%%%%%%%%%%%%%%%%%%%%%%%%%%%%%%%%%%
%%%%%%%%%%%%%%%%%%%%%%%%%%%%%%%%%%%%%%%%%%%%%%%%%%%%%%%%

%%%%%%%  END OF STANDARD STUFF %%%%%%%%%

%%%%%%%%%%%%%%%%%%%%%%%%%%%%%%%%%%%%%%%%%%%%%%%%%%%%%%%%
%%%%%%%%%%%%%%%%%%%%%%%%%%%%%%%%%%%%%%%%%%%%%%%%%%%%%%%%
%%%%%%%%%%%%%%%%%%%%%%%%%%%%%%%%%%%%%%%%%%%%%%%%%%%%%%%%
%%%%%%%%%%%%%%%%%%%%%%%%%%%%%%%%%%%%%%%%%%%%%%%%%%%%%%%
%%%%%%%%%%%
%%% Spell

\hyphenation{re-pub-lished}

\mathsurround=1pt

\def\bfdefault{b}

\def \F{{\mathbb F}}
\def \K{{\mathbb K}}
\def \Z{{\mathbb Z}}
\def \Q{{\mathbb Q}}
\def \R{{\mathbb R}}
\def \C{{\mathbb C}}
\def \N{{\mathbb N}}
\def\Fp{\F_p}
\def \fp{\Fp^*}

\def\Kmnp{\cK_p(m,n)}
\def\Kmnq{\cK_q(m,n)}
\def\KKap{\cH_p(a)}
\def\KKaq{\cH_q(a)}
\def\KKmnp{\cH_p(m,n)}
\def\KKmnq{\cH_q(m,n)}

\def\Kl{{\mathsf K}}

\def\Klmnp{\cK_p(\ell, m,n)}
\def\Klmnq{\cK_q(\ell, m,n)}

\def \SALMNq {\cS_q(\balpha;\cL,\cI,\cJ)}
\def \SALMNp {\cS_p(\balpha;\cL,\cI,\cJ)}

\def \SACXMQX {\fS(\balpha,\bzeta, \bxi; M,Q,X)}

\def \balpha{\bm{\alpha}}
\def \bbeta{\bm{\beta}}
\def \bgamma{\bm{\gamma}}
\def \blambda{\bm{\lambda}}
\def \bchi{\bm{\chi}}
\def \bphi{\bm{\varphi}}
\def \bpsi{\bm{\psi}}

\def\SIJq{S_q(\balpha; \cI,\cJ)}
\def\SIJp{S_p(\balpha; \cI,\cJ)}

\def\UMNp{U_p(\balpha, \bbeta; M,N)}
\def\VMNp{V_p(\balpha, \bbeta; M,N)}
\def\LG#1{\(\frac{#1}{q}\)}

\def\SMJq{S_q(\balpha; \cM,\cJ)}
\def\SMJp{S_p(\balpha; \cM,\cJ)}

\def\WIKq{W_{a,q}(\bgamma; \cI,K)}
\def\WIKp{W_{a,p}(\bgamma; \cI,K)}
\def\WMKq{W_{a,q}(\bgamma; \cM,K)}
\def\WMKp{W_{a,p}(\bgamma; \cM,K)}

\def\WaIKq{W_{a,q}(\balpha, \bgamma; \cI;K)}
\def\WaIKp{W_{a,p}(\balpha, \bgamma; \cI,K)}
\def\WaMKq{W_{a,q}(\balpha, \bgamma; \cM;K)}
\def\WaMKp{W_{a,p}(\balpha, \bgamma; \cM;K)}
\def\WaMKKq{\widetilde{W}_{a,q}(\balpha, \bgamma; \cM,\cK)}
\def\WaMKKp{\widetilde{W}_{a,p}(\balpha, \bgamma; \cM,\cK)}

\def\WaIKKq{\widetilde{W}_{a,q}(\balpha, \bgamma; \cI,\cK)}
\def\WaIKKp{\widetilde{W}_{a,p}(\balpha, \bgamma; \cI,\cK)}

\def\WAMNQ{W_q(\balpha, \bgamma; M,N,Q)}

\def\RIJp{\cR_p(\cI,\cJ)}
\def\RIJq{\cR_q(\cI,\cJ)}

\def\TWXJp{\cT_p(\bomega;\cX,\cJ)}
\def\TWXJq{\cT_q(\bomega;\cX,\cJ)}
\def\TWpXJp{\cT_p(\bomega_p;\cX,\cJ)}
\def\TWqXJq{\cT_q(\bomega_q;\cX,\cJ)}
\def\TWJq{\cT_q(\bomega;\cJ)}
\def\TWqJq{\cT_q(\bomega_q;\cJ)}

\newcommand{\supp}{\operatorname{supp}}

  \def \kbar{k^{-1} }
 \def \xbar{x^{-1} }
  \def \ybar{y^{-1} }

\title[Squarefree polynomials]{Squarefree values of random polynomials}
\author{Bryce Kerr}
\address{UNSW Canberra School of Science, Australia}
\email{bryce.kerr@unsw.edu.au}
\author{Watcharakiete Wongcharoenbhorn}
\address{UNSW Canberra School of Science, Australia}
\email{w.wongcharoenbhorn@unsw.edu.au}

\begin{abstract}
In this paper we consider the error term in the expected count of squarefree values of a polynomial when evaluated at integers \(n\le N\). We prove a power saving on average over polynomials of fixed degree and height at most \(H\). Our results improve previous estimates of Browning and Shparlinski in that the required range of \(H\) is independent of the degree, but depends on \(N\)
\end{abstract}
\maketitle
\tableofcontents

\section{Introduction}
For $f\in\Z[X]$, let
\[
 S_f(N)=\#\{1\le n\le N: f(n)\text{ is squarefree}\}.
\]
Here a nonzero integer is squarefree if its absolute value is squarefree.
Erd\H{o}s conjectured that an irredicuble polynomial $f$
with no fixed prime-square divisor takes infinitely many squarefree
values~\cite{Erdos}.  A folklore strengthening of this conjecture is
that the squarefree values have a positive density. To state this
precisely, define
\[
 \rho_f(q)=\#\{a\in\Z/q\Z:f(a)\equiv0\pmod q\},
 \qquad
 c_f=\prod_p\left(1-\frac{\rho_f(p^2)}{p^2}\right).
\]
The conjecture then predicts that for every irreducible
polynomial $f$
\begin{equation}
\label{eq:conjecture}
 S_f(N)=c_fN+o(N).
\end{equation}
The condition $\rho_f(p^2)<p^2$ for every prime $p$ is equivalent to
the absence of a fixed prime-square divisor, and it implies that
$c_f>0$. Thus~\eqref{eq:conjecture} can be considered a strengthening of Erd\H{o}s' conjecture.

Estermann~\cite{Estermann} established~\eqref{eq:conjecture} for
quadratic polynomials. For cubics, Erd\H{o}s~\cite{Erdos} proved
infinitude and Hooley~\cite{Hooley} obtained~\eqref{eq:conjecture}.
Granville~\cite{Granville} proved~\eqref{eq:conjecture} for all irreducible polynomials assuming the
$abc$-conjecture. The general unconditional problem remains open. A recent OpenAI preprint
claims a proof of~\eqref{eq:conjecture} for irreducible quartics
\cite{OpenAI}.

We also mention work of Helfgott~\cite{Helfgott} who proved the corresponding asymptotic for squarefree values of cubic
polynomials when evaluated at primes. See also Reuss~\cite{Reuss} for subsequent refinements.

The problem becomes easier if we allow an averaging over the coefficients of $f$. Filaseta~\cite{Filaseta}
proved that almost all polynomials attain a squarefree value.
Quantitative estimates for the error in~\eqref{eq:conjecture} were obtained by
Shparlinski~\cite{Shp} and Browning and Shparlinski~\cite{BrShp}.
To state these results, put
\begin{align}
\label{eq:Fk}
 \cF_k(H)=\left\{\sum_{j=0}^k a_jX^j:
 a_j\in\Z,\ |a_j|\le H,\ \gcd(a_0,\ldots,a_k)=1\right\}.
\end{align}
Browning and Shparlinski~\cite[Theorem~1.1]{BrShp} proved that
\begin{equation}
\label{eq:bs}
 \frac{1}{\#\cF_k(H)}\sum_{f\in\cF_k(H)}|S_f(N)-c_fN|
 \ll N^{1-\delta}
\end{equation}
for some $\delta>0$, provided that $k\ge4$ and
\begin{equation}
\label{eq:bs-conds}
 N^{k-3+\varepsilon}\le H\le N^A.
\end{equation}
Jelinek~\cite{Jelinek} obtained further almost all asymptotics when averaged over subsets of~\eqref{eq:Fk}
with all $a_j$'s fixed except two.

Our main result improves the range of $H$ in
\eqref{eq:bs-conds} by a bound independent of $k$.
\begin{theorem}
\label{thm:main}
Let $k\ge4$ and $A\ge1$ be fixed, let $\varepsilon\in(0,1/2)$, and put
$\varpi=(3-\sqrt5)/2$. If
\begin{align}
\label{eq:H11}
 N^{\varpi+\varepsilon}\le H\le N^A,
\end{align}
then
\begin{equation}
\label{eq:mainthm}
 \frac{1}{\#\cF_k(H)}
 \sum_{f\in\cF_k(H)}|S_f(N)-c_fN|
 \ll_{k,\varepsilon,A} N^{1-\delta},
\end{equation}
where $\delta>0$ depends only on $\varepsilon$.
\end{theorem}
\begin{rem}
In Theorem~\ref{thm:main}, one may take $\delta=\varepsilon/7$.
\end{rem}
For $H$ satisfying~\eqref{eq:H11}, all but a proportion $O(N^{-\delta/2})$ of the polynomials
in $\cF_k(H)$ satisfy $S_f(N)=c_fN+O(N^{1-\delta/2})$.

Sofos has studied the second moment of the error in~\eqref{eq:conjecture}. From~\cite[Theorem~1.1]{Sofos} it follows that for every fixed $\varepsilon>0$ and $H\ge N^{k+3+\varepsilon}$, we have
\[
 (2H)^{-k-1}
 \sum_{\substack{f=\sum_{j=0}^k a_jX^j\\
                  |a_j|\le H,\ a_k\ne0}}
 |S_f(N)-c_fN|^2
 =\gamma_0N^{1/2}+O(N^{1/2-\delta})
\]
for a constant $\gamma_0>0$ and some $\delta>0$ depending on $\varepsilon$.

More recently, Simhi~\cite{Simhi} evaluated higher-order moments of the normalized error terms in the range $N=H^{o(1)}$ and deduced that they are distributed according to a Gaussian mixture that is not itself Gaussian.
\newline

{\bf Acknowledgement:} The authors thank ChatGPT for editing and proofreading.

\section{Notation}
As usual, $\mu$ denotes the M\"{o}bius function, so that $\mu^2$ is the indicator function of squarefree numbers. Define
\[
 \cF_k^*(H)=\{f\in\cF_k(H):\deg f=k,\ f\text{ is irreducible over }\Z\}.
\]

For $q\in\N$ and real numbers $0\le M<N$, write

\begin{equation}
\label{eq:rho_f-definition}
\begin{aligned}
 \rho_f(q,N)&=\#\{n\in\N:n\le N,\ f(n)\equiv0\pmod q\},\\
 \rho_f(q,M,N)&=\#\{n\in\N:M<n\le N,\ f(n)\equiv0\pmod q\}.
\end{aligned}
\end{equation}
 Throughout, $k\ge4$ and
$A\ge1$ are fixed. Implied constants may depend on these parameters
and on fixed positive parameters such as $\varepsilon$ and $\eta$.
All occurrences of $o(1)$ are uniform in the stated ranges as
$N\to\infty$, with $H\le N^A$. 
\section{Proof overview}
Following Browning and Shparlinski~\cite{BrShp}, M\"obius inversion and equidistribution of integers in residue classes extract the main term of
$c_fN$ in~\eqref{eq:conjecture}. The resulting error is controlled by the following averaged lattice point count
\begin{equation}
\label{eq:SDdef}
 S(M,D)=\frac{1}{H^{k+1}}
 \sum_{\substack{d\sim D\\n\sim M}}
 \#\bigl(\cL(n,d^2)\cap[-H,H]^{k+1}\bigr),
\end{equation}
where
\begin{align}
\label{eq:LDB}
 \cL(n,d^2)=
 \{(h_0,\ldots,h_k)\in\Z^{k+1}:
 h_0+h_1n+\cdots+h_kn^k\equiv0\pmod{d^2}\},
\end{align}
with $M,D$ satisfying
\begin{align*}
N^{1-\eta}\le M\le N, \quad N\le D \le (HN^{k})^{1/2}.
\end{align*}
For the above parameter ranges, we need to show
$$S(M,D)\ll N^{1-\delta}.$$
Our proof is split into cases depending on the size of $H$ relative to $N$ and the size of $D$ relative to both $H$ and $N$.

Consider first $H>N$. For large $D$, we appeal to a result of Browning and Shparlinski based on the determinant method; see Lemma~\ref{lem:det-met} below. For medium-sized $D$, we use transference theorems to pass to the dual lattice; see Lemma~\ref{lem:lat-cont-large-H} below. This reduces to counting solutions to the following system of congruences
\begin{align}
\label{eq:hgsdflksa}
h_0n^j\equiv h_j\pmod{d^2}, \quad 1\le j \le k
\end{align}
with variables 
\begin{align}
\label{eq:sdgkdfglkdjgdslkjgsdflkgsdkljgh}
|h_j|\le \frac{D^2}{H}, \quad n\sim N, \quad d\sim D.
\end{align}
Rearranging variables in~\eqref{eq:hgsdflksa} (and using that $N<H$) allows us to lift the congruences to equations over $\Z$. In particular, the successive coordinates of $(h_0,h_1,\dots,h_k)$ satisfy the following linear recurrence:
\[
 h_{j+1}=nh_j\quad 0\le j<k.
\]
Thus $|h_k|=n^k|h_0|$, and the upper bounds on $h_k$ and $n$ in~\eqref{eq:sdgkdfglkdjgdslkjgsdflkgsdkljgh} leave only a few possible
values of $h_0$, since 
$$h_0\ll \frac{h_k}{n^{k}}\ll \frac{D^2}{HN^{k}}.$$
Fixing $h_0,n,d$ then leaves $O(1)$ solutions for the remaining variables $h_1,\dots,h_k$.

Consider next $H\le N$. For large $D$, we exploit the averaging over coefficients of $f$ to reduce the problem to counting squares in short intervals; see Lemma~\ref{lem:small-coef-control}. This improves on Browning and Shparlinski's use of the determinant method. In particular, writing 
\begin{align}
f(n)=sd^2
\end{align}
with 
$$f(n)=a_0+a_1n+\dots+a_kn^{k}$$
and fixing $n,s$ and the leading coefficients of $f$, corresponding to $a_k,a_{k-1},\dots,a_m$, 
each solution in the remaining variables $a_0,a_1,\dots,a_{m-1}$ corresponds to a square in a short interval determined by $s,n,a_k,\dots,a_m$. These squares can be counted by elementary methods.

Medium-sized $D$ are handled in Lemma~\ref{lem:lat-cont-small-H}, and our main input is a lattice rescaling inequality
\begin{align}
\label{eq:lslslslslslslsl}
 \#(\cL\cap\cB)\ll
 \frac{\operatorname{vol}(\cB)}{\det\cL}
 +t^{-k}\#(\cL\cap t\cB)
 \qquad(0<t\le1).
\end{align}
Applying~\eqref{eq:lslslslslslslsl} in~\eqref{eq:SDdef} and then interchanging the sums reduces the problem to counting divisors of $f(n)$, which we handle with standard estimates. The saving of one power of $t$ over the trivial bound $t^{-(k+1)}$ in~\eqref{eq:lslslslslslslsl} is essential.
This replaces the use of Davenport's lemma in the
related argument of Browning and Shparlinski~\cite[Lemma~3.2]{BrShp} and avoids the loss associated with directly estimating the successive minima of our lattices in~\eqref{eq:SDdef}.

\section{Preliminary manipulations -- small square divisors}
\label{sec:prelim}
We first use the standard reduction via M\"obius inversion. Put
\[
 B=\bigl((k+1)HN^k\bigr)^{1/2}.
\]
\begin{lemma}
\label{lem:prelim}
Let $0<\eta<1/7$ and $N^\eta\le H\le N^A$. Then
\begin{align}
\label{eq:first-reduc}
 &\frac{1}{\#\cF_k(H)}
 \sum_{f\in\cF_k(H)}|S_f(N)-c_fN|
 \ll N^{1-\eta+o(1)}\nonumber\\
 &\quad+\frac{1}{\#\cF_k^*(H)}
 \sum_{f\in\cF_k^*(H)}
 \sum_{N^{1-\eta}<d\le B}
 \mu^2(d)\rho_f(d^2,N^{1-\eta},N).
\end{align}
\end{lemma}
\begin{proof}
There are $O_k(H^k)$ polynomials of degree less than $k$.
By Kuba~\cite[Theorem~1]{Kuba}, there are also $O_k(H^k)$
reducible primitive polynomials of degree $k$. From
$$\#\cF_k(H)\asymp H^{k+1}$$ it follows
\begin{equation}
\label{eq:reduc-reducible-polyn}
 \frac{1}{\#\cF_k(H)}
 \sum_{f\in\cF_k(H)\setminus\cF_k^*(H)}
 |S_f(N)-c_fN|\ll \frac NH
\end{equation}
and we have used the fact that $0\le S_f(N)\le N$ and $c_f\ll 1$.

For $f\in\cF_k^*(H)$ and $1\le n\le N$, we have
$0<|f(n)|\le B^2$. Hence
\begin{equation}
\label{eq:Sfmobius}
 S_f(N)=\sum_{1\le n\le N}\sum_{d^2\mid f(n)}\mu(d)
       =\sum_{d\le B}\mu(d)\rho_f(d^2,N).
\end{equation}
Writing
\[
 M_f(N)=\sum_{d\le N^{1-\eta}}\mu(d)\rho_f(d^2,N),
\]
we obtain
\[
 S_f(N)=M_f(N)+
 O\left(\sum_{N^{1-\eta}<d\le B}\mu^2(d)\rho_f(d^2,N)\right).
\]

By~\cite[Equation~(4.8)]{Shp}, as recorded in~\cite[Equation~(4.1)]{BrShp},
\[
 M_f(N)=c_fN+
 O\bigl((N^{1-\eta}+N^\eta)H^{o(1)}\bigr)
       =c_fN+O(N^{1-\eta+o(1)}).
\]
Finally, $f(n)\ne0$ and the divisor bound give, uniformly for
$f\in\cF_k^*(H)$,
\begin{equation}
\label{eq:FFA!!!}
 \sum_{n\le N^{1-\eta}}\sum_{d^2\mid f(n)}1
 \ll N^{1-\eta+o(1)}.
\end{equation}
Thus 
$$\rho_f(d^2,N)=\rho_f(d^2,N^{1-\eta},N)+O(N^{1-\eta+o(1)})$$
which after combining with~\eqref{eq:reduc-reducible-polyn},~\eqref{eq:Sfmobius} then summing over $f\in \cF_k(H)$ proves the lemma.
\end{proof}

\section{Squares in short intervals -- large square divisors}
\label{subsec:disp-large-square-div}
For $0<\eta<1/7$, define
\begin{equation}
\label{eq:REdef}
 R(E)=\frac{1}{\#\cF_k^*(H)}
 \sum_{f\in\cF_k^*(H)}
 \sum_{E<d\le B}\mu^2(d)\rho_f(d^2,N^{1-\eta},N).
\end{equation}
An empty sum is understood to be zero. With $N,H$ as in
Theorem~\ref{thm:main}, we take $\eta=2\varepsilon/7$ and
\begin{equation}
\label{eq:Edef}
 E=
 \begin{cases}
 H^{k+1}N^{-(k+2)\eta},& H\le N^{2/3},\\
 H^{1/2}N^{(k-1)/2+\eta},& H>N^{2/3}.
 \end{cases}
\end{equation}
Both choices satisfy $E\ge N^{1-\eta}$ for sufficiently large $N$.

\begin{prop}
\label{prop:large-squares}
Assume the hypotheses of Theorem~\ref{thm:main}, put
$\eta=2\varepsilon/7$, and choose $E$ as in~\eqref{eq:Edef}. Then
\begin{align}
\label{eq:second-reduc}
 &\frac{1}{\#\cF_k(H)}
 \sum_{f\in\cF_k(H)}|S_f(N)-c_fN|
 \ll N^{1-\eta+o(1)}\nonumber\\
 &\quad+\frac{1}{\#\cF_k^*(H)}
 \sum_{f\in\cF_k^*(H)}
 \sum_{N^{1-\eta}<d\le\min(E,B)}
 \mu^2(d)\rho_f(d^2,N^{1-\eta},N).
\end{align}
\end{prop}

The following is based on a reduction to counting squares in short intervals and will be used for large $D$ and small $H$.
\begin{lemma}[Large $D$ and small $H$]
\label{lem:small-coef-control}
Let $0<\eta<1/7$, $H=N^h$ with $\eta\le h\le2/3$, and
$N^{1-\eta}\le E<B$. Put $S=B^2/E^2$. Then
\begin{equation}
\label{eq:err-frm-large-d-small-H}
 R(E)\ll
 N^{1+o(1)}
 H^{-k/2+h+\eta(k/2+1)}
 S^{1-h/2}.
\end{equation}
\end{lemma}

For larger $D$, we will appeal to a result of Browning and Shparlinski based on the determinant method ~\cite[Lemma~2.5]{BrShp}.
\begin{lemma}[Individual-polynomial bound]
\label{lem:large-coef-control}
Let $0<\eta<1/7$, $N^\eta\le H\le N^A$ and $1\le E<B$. Then
\begin{equation}
\label{eq:err-frm-large-d-large-H}
 R(E)\ll
 \left(\frac{H^{1/2}N^{(k+1)/2}}{E}
       +\frac{HN^k}{E^2}\right)N^{o(1)}.
\end{equation}
\end{lemma}

\begin{proof}[Proof of Proposition~\ref{prop:large-squares}]
Before proving Lemma~\ref{lem:small-coef-control} and Lemma~\ref{lem:large-coef-control}, we show how they imply Proposition~\ref{prop:large-squares}.

By Lemma~\ref{lem:prelim}, it suffices to prove
$R(E)\ll N^{1-\eta+o(1)}$. There is nothing to prove if $E\ge B$.
If $H>N^{2/3}$, Lemma~\ref{lem:large-coef-control} gives
\[
 R(E)\ll N^{1-\eta+o(1)}+N^{1-2\eta+o(1)}.
\]
Suppose now that $H=N^h\le N^{2/3}$. Write
$e=(k+1)h-(k+2)\eta$, so that $E=N^e$. Since
$h\ge\varpi+\varepsilon$, we have
\[
 e\ge(k+1)\varpi+\frac{5k+3}{7}\varepsilon>1.
\]
Lemma~\ref{lem:small-coef-control} gives
$R(E)\ll N^{F(h)+o(1)}$, where
\begin{equation}
\label{eq:small-tail-exponent}
 F(h)=1+k(1-3h+h^2)-h+\frac32h^2
          +\eta(k+2)\left(2-\frac h2\right).
\end{equation}
For $\varpi+\varepsilon\le h\le2/3$, we have
\[
 1-3h+h^2\le-\frac53\varepsilon,
 \qquad -h+\frac32h^2\le0,
 \qquad (k+2)\left(2-\frac h2\right)\le3k.
\]
The first inequality follows from $1-3\varpi+\varpi^2=0$ and
$2h-3\le-5/3$. Therefore
\[
 F(h)\le1-\frac53k\varepsilon+3k\eta
       =1-\frac{17}{21}k\varepsilon
       \le1-\eta.
\]
This proves the required bound.
\end{proof}

\subsection{Small values of $H$: Proof of Lemma~\ref{lem:small-coef-control}}
\begin{proof}
If $d>E$ and $d^2\mid f(n)$, write $|f(n)|=sd^2$ with
$1\le s\le S$.
Divide $1\le s\le S$ into $O(\log N)$ intervals
$U\le s<2U$, where $U=N^\theta\ge1$.
For each such interval put
\begin{equation}
\label{eq:mdef}
 m=\left\lfloor
      \frac{k}{2}+\frac{\theta}{2}-h
      -\eta\left(\frac{k}{2}+1\right)
    \right\rfloor.
\end{equation}
We have $0\le m\le k-1$ for sufficiently large $N$.
Indeed, the expression inside the floor is at least
$k/2-2/3-(k/2+1)/7>0$. Also,
\[
 \theta\le h+k-2+2\eta+O(1/\log N)
\]
by $E\ge N^{1-\eta}$, so the same expression is at most
$k-1-h/2-\eta k/2+O(1/\log N)<k-1$.

Fix $n,s,a_{m+1},\ldots,a_k$, with
$N^{1-\eta}<n\le N$, $U\le s<2U$, $|a_j|\le H$ and $a_k\ne0$.
Write
\begin{equation}
\label{eq:fsquares}
 f(X)=a_0+a_1X+\cdots+a_kX^k,
 \qquad
 \widetilde f(X)=a_{m+1}X^{m+1}+\cdots+a_kX^k.
\end{equation}
Since $H\le N^{2/3}$ and $\eta<1/7$, we have $H/n=o(1)$.
Consequently,
\begin{equation}
\label{eq:tildefbounds}
 |\widetilde f(n)|\ge\frac{n^k}{2},
 \qquad
 |f(n)-\widetilde f(n)|\ll Hn^m,
\end{equation}
and every admissible $f(n)$ has the sign of $a_k$.
The possible values of $\sqrt{|f(n)|/s}$ lie in an interval of length
\begin{align}
\label{eq:square-interval-length}
 \ll Hn^{m-k/2}U^{-1/2}
 &\ll N^{h+m-k/2+\eta k/2-\theta/2}\nonumber\\
 &\ll N^{-\eta}.
\end{align}
In the second inequality we used $0\le m\le k-1$ and
$N^{1-\eta}<n\le N$; the last follows from~\eqref{eq:mdef}.
For large $N$, this interval contains at most one positive integer $d$. For fixed $n,s,a_{m+1},\dots,a_k$, the value of $d$ fixes $f(n)$. If two coefficient tuples $a_0,\dots,a_k$ and $a_0',\dots,a_k'$ give this same value, then
\[
 \sum_{j=0}^m(a_j-a'_j)n^j=0.
\]
Since $H/n=o(1)$, we have $2H<n$ for sufficiently large $N$. Reducing the displayed identity modulo $n$ and then dividing successively by $n$ shows that all the coefficient differences vanish. Thus there is at most one tuple
$(a_0,\ldots,a_m)$ for each choice of the fixed variables $n,s,a_k,\dots,a_{m+1}$.

There are $O(H^{k-m}NU)$ ways to fix $n,s,a_k,\dots,a_{m+1}$. The result follows after some rearranging, since 
\begin{align}
\label{eq:HNS}
H^{k-m}NU \ll H^{k+1}NUH^{-m-1}
 &\ll H^{k+1}NH^{-k/2+h+\eta(k/2+1)}U^{1-h/2}.
\end{align}
Here we used $m+1>k/2+\theta/2-h-\eta(k/2+1)$ and
$H^{-\theta/2}=U^{-h/2}$. Summing the dyadic intervals and using
$U\le S$ and $1-h/2>0$ proves~\eqref{eq:err-frm-large-d-small-H}.
\end{proof}

\subsection{Larger values of $H$: Proof of Lemma~\ref{lem:large-coef-control}}
We use the following consequence of
Browning and Shparlinski~\cite[Lemma~2.5]{BrShp}.
\begin{lemma}
\label{lem:det-met}
Let $f\in\cF_k^*(H)$ and $S\ge1$. Then
\[
 \#\{(n,d,s)\in\N^3:n\le N,\ s\le S,\ f(n)=sd^2\}
 \ll (N^{1/2}S^{1/2}+S)(HN)^{o(1)}.
\]
\end{lemma}
\begin{proof}[Proof of Lemma~\ref{lem:large-coef-control}]
For $f\in\cF_k^*(H)$ we have $f(n)\ne0$ for each $n\le N$. If $d>E$ and
$d^2\mid f(n)$, then $|f(n)|=sd^2$ for some
$1\le s\le (k+1)HN^k/E^2$. The desired result follows after applying Lemma~\ref{lem:det-met}
 then averaging over $\cF_k^*(H)$.

\end{proof}

\section{Background on the geometry of numbers}
\label{sec:geometry}
All lattices in this section have full rank, and all convex bodies
are symmetric about the origin and have nonempty interior. We write $\operatorname{vol}$ for Lebesgue measure.
For a lattice $\cL\subseteq\R^r$ and a convex body $\cB$, define
\[
 \lambda_i=\inf\{\lambda>0:
 \cL\cap\lambda\cB\text{ contains $i$ linearly independent points}\}.
\]
Minkowski's second theorem gives the following estimate;
see~\cite[Theorem~3.30]{TaoVu}.
\begin{lemma}
\label{lem:mst}
With notation as above,
\[
 \frac{1}{\lambda_1\cdots\lambda_r}
 \asymp_r\frac{\operatorname{vol}(\cB)}{\det\cL}.
\]
\end{lemma}
We also use the standard lattice-point estimate
\cite[Exercise~3.5.6]{TaoVu}.
\begin{lemma}
\label{lem:lattice}
With notation as above,
\[
 \#(\cL\cap\cB)\asymp_r
 \prod_{j=1}^r\max\{1,\lambda_j^{-1}\}.
\]
\end{lemma}
The next result will be used to rescale our lattice point counts.
\begin{lemma}
\label{lem:scales}
For $0<t\le1$,
\[
 \#(\cL\cap\cB)\ll
 \frac{\operatorname{vol}(\cB)}{\det\cL}+
 t^{-(r-1)}\#(\cL\cap t\cB).
\]
\end{lemma}

\begin{proof}
If $\lambda_r\le1$, Lemmas~\ref{lem:mst} and~\ref{lem:lattice} give
$$\#(\cL\cap\cB)\ll\operatorname{vol}(\cB)/\det\cL.$$
If $\lambda_r>1$, then by Lemma~\ref{lem:lattice}
\[
 \#(\cL\cap\cB)\asymp_r
 \prod_{j=1}^{r-1}\max\{1,\lambda_j^{-1}\}.
\]
Using $\max(1,\lambda_j^{-1})\le t^{-1}\max(1,t/\lambda_j)$
for $j<r$, and noting that the $j$-th successive minima of $\cL$ with respect to $t\cB$ is $\lambda_j/t$,
we get
\[
 \#(\cL\cap\cB)
 \ll_r t^{-(r-1)}
 \prod_{j=1}^{r-1}\max(1,t/\lambda_j)
 \ll_r t^{-(r-1)}\#(\cL\cap t\cB).
\]
\end{proof}

Define the dual lattice and polar body by
\begin{equation}
\label{eq:Gammastar}
 \cL^*=\{y\in\R^r:\langle y,z\rangle\in\Z
                       \text{ for every }z\in\cL\}
\end{equation}
and
\begin{equation}
\label{eq:Dstar}
 \cB^*=\{y\in\R^r:|\langle y,z\rangle|\le1
                       \text{ for every }z\in\cB\}.
\end{equation}
We use the following transference estimate; see Banaszczyk~\cite{Ba}.
\begin{lemma}[Transference]
\label{lem:transfer}
If $\lambda_j^*$ are the successive minima of $\cL^*$ with respect
to $\cB^*$, then
\[
 \lambda_j\lambda_{r-j+1}^*\ll_r1\qquad(1\le j\le r).
\]
\end{lemma}

\begin{lemma}
\label{lem:latticedual}
With notation as above
\[
 \#(\cL\cap\cB)\ll_r
 \frac{\operatorname{vol}(\cB)}{\det\cL}
 \#(\cL^*\cap\cB^*).
\]
\end{lemma}
\begin{proof}
By Lemmas~\ref{lem:mst}, \ref{lem:lattice}, and \ref{lem:transfer}
\[
 \#(\cL\cap\cB)
 \ll_r\frac{\prod_j\max(1,\lambda_j)}{\prod_j\lambda_j}
 \ll_r\frac{\operatorname{vol}(\cB)}{\det\cL}
        \prod_j\max(1,(\lambda_j^*)^{-1}).
\]
The result then follows from another application of Lemma~\ref{lem:lattice}.
\end{proof}

We now specialize to the lattices $\cL(n,d^2)$ defined in~\eqref{eq:LDB}.
For these lattices, the dual has the following explicit description; compare
the explicit dual-lattice formula for a related congruence lattice in~\cite{BK}.
\begin{lemma}
\label{lem:dual-lattice-explicit}
For $n\in\Z$ and $d\in\N$,
\begin{align*}
 \cL(n,d^2)^*
 &=\Z^{k+1}+\frac1{d^2}\Z(1,n,\ldots,n^k)\\
 &=\left\{\frac{(h_0,\ldots,h_k)}{d^2}:
 h_j\in\Z,\ h_j\equiv h_0n^j\pmod{d^2}\ (1\le j\le k)\right\}.
\end{align*}
\end{lemma}
\begin{proof}
This follows from the fact that the homomorphism $\Z^{k+1}\to\Z/d^2\Z$ given by
$(h_0,\ldots,h_k)\mapsto\sum_jh_jn^j$ is surjective and has kernel
$\cL(n,d^2)$.
\end{proof}

\begin{lemma}
\label{lem:boxdual}
For $\cB=[-H,H]^{k+1}$,
\[
 \cB^*=\{y\in\R^{k+1}:H\sum_{j=0}^k|y_j|\le1\}
       \subseteq[-H^{-1},H^{-1}]^{k+1}.
\]
\end{lemma}
\begin{proof}
This follows from the fact that the maximum of $|\langle y,z\rangle|$ over $z\in\cB$ is
$H\sum_j|y_j|$.
\end{proof}

\begin{lemma}
\label{lem:det}
For $n\in\Z$ and $d\in\N$, $\det\cL(n,d^2)=d^2$.
\end{lemma}
\begin{proof}
The surjective homomorphism used in the proof of
Lemma~\ref{lem:dual-lattice-explicit} shows that
$[\Z^{k+1}:\cL(n,d^2)]=d^2$.
\end{proof}

\section{Lattice counts -- medium square divisors}
\label{sec:lat-cont}
Recall that
\begin{equation}
\label{eq:SDM}
 S(D,M)=\frac1{H^{k+1}}
 \sum_{\substack{n\sim M\\d\sim D}}
 \#\bigl(\cL(n,d^2)\cap[-H,H]^{k+1}\bigr).
\end{equation}
We next estimate $S(D,M)$ using the geometry-of-numbers estimates from Section~\ref{sec:geometry}.

\begin{lemma}[Large-height lattice count]
\label{lem:lat-cont-large-H}
If $H\ge M\ge1$ and $D\ge1$, then
\[
 S(D,M)\ll
 \frac{MD}{H^{k+1}}+\frac MD+\frac{D}{HM^{k-1}}.
\]
\end{lemma}
\begin{proof}
Lemmas~\ref{lem:latticedual}, \ref{lem:boxdual} and~\ref{lem:det} give
\[
 S(D,M)\ll\frac1{D^2}
 \sum_{\substack{n\sim M\\d\sim D}}
 \#\bigl(\cL(n,d^2)^*\cap[-H^{-1},H^{-1}]^{k+1}\bigr).
\]
Apply Lemma~\ref{lem:scales} with $t=1/100$.
By Lemma~\ref{lem:det}, this yields
\begin{align}
\label{eq:dualrescaling}
 S(D,M)\ll\frac{MD}{H^{k+1}}+\frac1{D^2}
 \sum_{\substack{n\sim M\\d\sim D}}
 \#\bigl(\cL(n,d^2)^*\cap[-(100H)^{-1},(100H)^{-1}]^{k+1}\bigr).
\end{align}
By Lemma~\ref{lem:dual-lattice-explicit}, each point in the last
count corresponds to integers $w_0,\ldots,w_k$ satisfying
\begin{equation}
\label{eq:w0n}
 w_j\equiv w_0n^j\pmod{d^2}\quad(1\le j\le k),
 \qquad |w_j|\le\frac{d^2}{100H}\quad(0\le j\le k).
\end{equation}
For $0\le j<k$ we have $nw_j\equiv w_{j+1}\pmod{d^2}$ and
\[
 |nw_j-w_{j+1}|
 \le\frac{(2M+1)d^2}{100H}<d^2.
\]
Thus $w_{j+1}=nw_j$ for every $0\le j<k$, and
$|w_0|\ll D^2/(HM^k)$. For fixed $n,d$, this leaves
$O(1+D^2/(HM^k))$ possibilities, each determining all the $w_j$.
Substitution in~\eqref{eq:dualrescaling} proves the result.
\end{proof}

\begin{lemma}[Rescaled lattice count]
\label{lem:lat-cont-small-H}
Let $1\le M\le N$, $D\ge1$, $1\le H\le N^A$ and $0<t\le1$.
If $2tH<M$, then
\[
 S(D,M)\ll\frac MD+\frac{MD}{t^kH^{k+1}}+tMN^{o(1)}.
\]
\end{lemma}
\begin{proof}
Lemmas~\ref{lem:scales} and~\ref{lem:det} give
\begin{align}
\label{eq:SDM11}
 S(D,M)\ll\frac MD+
 \frac1{t^kH^{k+1}}
 \sum_{\substack{n\sim M\\d\sim D}}
 \#\bigl(\cL(n,d^2)\cap[-tH,tH]^{k+1}\bigr).
\end{align}
The zero vector contributes $O(MD)$. If $tH<1$, it is the only
integer vector in the box, and the result follows.
Otherwise there are $O((tH)^{k+1})$ nonzero integer vectors.
For each such vector, let $r$ be the largest index with $h_r\ne0$
and put $P(X)=\sum_{j=0}^rh_jX^j$. Since $n>M>2tH$, for $r\ge1$
\[
 |P(n)|\ge n^r-tH\sum_{j=0}^{r-1}n^j>0;
\]
the assertion is immediate for $r=0$.
Also $|P(n)|\ll_k HN^k$. Thus for each $n\sim M$ the number
of positive integers $d$ with $d^2\mid P(n)$ is $N^{o(1)}$.
The sum on the right hand side of~\eqref{eq:SDM11} is consequently at most
\[
 MD+(tH)^{k+1}MN^{o(1)}
\]
and this completes the proof.
\end{proof}

\section{Proof of Theorem~\ref{thm:main}}
\label{sec:completion-of-the-proof-of-main-theorem}
Put $\eta=2\varepsilon/7$, so that $0<\eta<1/7$, and choose
$E$ as in~\eqref{eq:Edef}. By Proposition~\ref{prop:large-squares},
dyadic subdivision, and $\#\cF_k^*(H)\asymp H^{k+1}$, it is sufficient to show that 
\begin{equation}
\label{eq:final-reduction}
 \max_{\substack{N^{1-\eta}\le M\le N\\
                 N^{1-\eta}\le D\le\min(E,B)}}S(D,M) \ll H^{k+1}N^{1-\eta+o(1)}.
\end{equation}
We estimate this maximum in three height ranges.

First suppose $H\le N$. Take $t=N^{-\eta}/4$ in
Lemma~\ref{lem:lat-cont-small-H}. Its hypothesis holds because
$2tH\le N^{1-\eta}/2<M$. Uniformly over the dyadic blocks,
\begin{equation}
\label{eq:primal-final}
 S(D,M)\ll N^\eta+
 \frac{N^{1+\eta k}E}{H^{k+1}}+N^{1-\eta+o(1)}.
\end{equation}
If $H\le N^{2/3}$, the middle term is exactly $N^{1-2\eta}$
with our choice of $E$. If $N^{2/3}<H\le N$, it is at most
\[
 N^{1+(k-1)/2+(k+1)\eta}H^{-k-1/2}
 \le N^{1-(k+5)/6+(k+1)\eta}
 \ll N^{1-\eta},
\]
since $\eta<1/7$ and $(k+5)/6>(k+2)/7$.
Thus~\eqref{eq:primal-final} is $O(N^{1-\eta+o(1)})$ throughout
$H\le N$.

Finally suppose $H>N$. Then $H\ge M$, so
Lemma~\ref{lem:lat-cont-large-H} gives
\[
 S(D,M)\ll
 \frac{NE}{H^{k+1}}+N^\eta+
 \frac{E}{HN^{(k-1)(1-\eta)}}.
\]
Substituting $E=H^{1/2}N^{(k-1)/2+\eta}$ and using $H>N$ bounds
the first and last terms by $N^{\eta-k/2}$ and
$N^{-k/2+k\eta}$, respectively. Hence this case also contributes
$O(N^{1-\eta+o(1)})$.

Equation~\eqref{eq:final-reduction} now gives
\[
 \frac1{\#\cF_k(H)}
 \sum_{f\in\cF_k(H)}|S_f(N)-c_fN|
 \ll N^{1-\eta+o(1)}.
\]
For sufficiently large $N$, all $o(1)$ terms can be absorbed
into $\eta/2$. This proves Theorem~\ref{thm:main} with
$\delta=\eta/2=\varepsilon/7$.


\begin{thebibliography}{99}
\bibitem{Ba}
W. Banaszczyk, Inequalities for convex bodies and polar reciprocal
lattices in $\R^n$, \emph{Discrete Comput. Geom.} \textbf{13} (1995),
217--231.

\bibitem{BK}
M. Bordignon and B. Kerr, An explicit P\'olya--Vinogradov inequality
via partial Gaussian sums, \emph{Trans. Amer. Math. Soc.}
\textbf{373} (2020), 6503--6527.

\bibitem{BrShp}
T. D. Browning and I. E. Shparlinski, Square-free values of random
polynomials, \emph{J. Number Theory} \textbf{261} (2024), 220--240.

\bibitem{Erdos}
P. Erd\H{o}s, Arithmetical properties of polynomials,
\emph{J. London Math. Soc.} \textbf{28} (1953), 416--425.

\bibitem{Estermann}
T. Estermann, Einige S\"atze \"uber quadratfreie Zahlen,
\emph{Math. Ann.} \textbf{105} (1931), 653--662.

\bibitem{Filaseta}
M. Filaseta, Squarefree values of polynomials,
\emph{Acta Arith.} \textbf{60} (1992), 213--231.

\bibitem{Granville}
A. Granville, ABC allows us to count squarefrees,
\emph{Int. Math. Res. Not.} (1998), no.~19, 991--1009.

\bibitem{Helfgott}
H. A. Helfgott, Square-free values of $f(p)$, $f$ cubic,
\emph{Acta Math.} \textbf{213} (2014), 107--135.

\bibitem{Hooley}
C. Hooley, On the power free values of polynomials,
\emph{Mathematika} \textbf{14} (1967), 21--26.

\bibitem{Jelinek}
P. Jelinek, Square-free values of polynomials on average,
preprint (2023), \href{https://arxiv.org/abs/2308.15146}{arXiv:2308.15146}.

\bibitem{Kuba}
G. Kuba, On the distribution of reducible polynomials,
\emph{Math. Slovaca} \textbf{59} (2009), 349--356.

\bibitem{OpenAI}
OpenAI, \emph{Squarefree values of quartics and power-free values of
polynomials}, preprint (2026), released 24 September 2026,
\href{https://github.com/openai/math/blob/main/preprints/Squarefree-values-of-quartics-and-power-free-values-of-polynomials-September-24-2026/manuscript.pdf}{OpenAI Math repository}.

\bibitem{Reuss}
T. Reuss, Power-free values of polynomials,
\emph{Bull. London Math. Soc.} \textbf{47} (2015), 270--284.

\bibitem{Shp}
I. E. Shparlinski, On the average number of square-free values of
polynomials, \emph{Canad. Math. Bull.} \textbf{56} (2013), 844--849.

\bibitem{Simhi}
O. Simhi, \emph{Squarefree values of random polynomials are Gaussian
mixtures}, preprint (2026),
\href{https://arxiv.org/abs/2610.29911}{arXiv:2610.29911}.

\bibitem{Sofos}
E. Sofos, Counting square-free values of random polynomials,
preprint (2026), \href{https://arxiv.org/abs/2601.19319}{arXiv:2601.19319}.

\bibitem{TaoVu}
T. Tao and V. Vu, \emph{Additive Combinatorics},
Cambridge Studies in Advanced Mathematics, vol.~105,
Cambridge University Press, Cambridge, 2006.

\end{thebibliography}
\end{document}